\documentclass[12pt]{amsart}
\usepackage[latin1]{inputenc}
\usepackage{epsfig,amsfonts}
\usepackage{amscd}
\usepackage[english]{babel}
\usepackage{graphicx,psfrag}
\usepackage{enumerate}
\usepackage{amsmath,amsfonts,amsthm,epsfig,latexsym,graphicx,amssymb}

\newtheorem*{Poincare}{Poincar\'{e}'s Theorem}
\newtheorem{cor}{Corollary}[section]

\newtheorem{prop}{Proposition}[section]
\newtheorem{thm}{Theorem}
\newtheorem{lem}{Lemma}[section]

\newtheorem{rem}{Remark}[section]

\newcommand{\fin}{\hfill{ \raggedleft \rule{1ex}{1ex}}}

 \def\RR{{\mathbb R}} 
\def\TT{{\mathbb T}}

\def\R{I\kern -0.37 em R}
\def\N{I\kern -0.37 em N}
\def\Z{I\kern -0.37 em Z}
\def\supess_#1{\mathop{\rm supess}\limits_{#1}}
\def\infess_#1{\mathop{\rm infess}\limits_{#1}}

\def\inte{\mathop{\rm int}}
\def\fr{\mathop{\rm frac}}
\def\mod{\mathop{\rm mod}}

\begin{document}

\title[Rotation set and entropy]{\bf Rotation set and entropy}

\author{ ENRICH Heber, GUELMAN Nancy, LARCANCH\'E  Audrey, LIOUSSE Isabelle}
\thanks{This paper was partially supported by the  Universidad de la
  Rep\'ublica,  Uruguay, the Universit\'{e} de Lille 1, France  and the PREMER project}
\address{{\bf   Audrey Larcanch\'{e}, Isabelle Liousse}, UMR CNRS 8524, Universit\'{e} de Lille1,
59655 Villeneuve d'Ascq C\'{e}dex,   France.  \emph {liousse@math.univ-lille1.fr},  \emph{larcanch@math.univ-lille1.fr}}
\address{{\bf  Heber Enrich, Nancy Guelman}
IMERL, Facultad de Ingenier\'{\i}a, Universidad de la Rep\'ublica,
C.C. 30, Montevideo, Uruguay.  \emph{nguelman@fing.edu.uy}, \emph {enrich@fing.edu.uy}}


\begin{abstract}
In 1991 Llibre and MacKay proved  that if  $f$ is  a  2-torus
homeomorphism isotopic to identity and the  rotation set of $f$  has
a non empty interior then $f$  has positive topological entropy.
Here, we give a converselike theorem. We show
that the interior of the  rotation set of a  2-torus  $C^{1+ \alpha}$
diffeomorphism isotopic to identity of positive topological entropy  is not
empty,  under  the additional hypotheses that $f$ is
 topologically transitive and irreducible.


\end{abstract}

\maketitle

\section{Introduction}

\subsection{History.}

The theory of dynamical systems began with  Henri Poincar\'{e}'s
approach to  studying  toral flows. It consists in passing to the
first return map on a topological circle. Hence,  the initial
requirement  is replaced by  a qualitative study of dynamical properties of  a circle map. Let $f \colon  {\mathbb R}/ {\mathbb Z}\to \mathbb R/\mathbb Z$ be a
circle homeomorphism and $ \tilde f \colon {\mathbb R}\to{\mathbb R}$ be
a lift of $f$. The Poincar\'{e}'s rotation number of $f$ is defined as
$$ \rho(f) = \lim_ {n\to +\infty} \frac { \tilde f ^n (x) - x } {n}
\ \ \ (\mod 1).$$  It's easy to see that this limit exists and  depends
neither  on the point $x$ in $\mathbb R$  nor on  the lift $\tilde f$ of $f$.

From the definition,  the formulas  $\rho(f^n) = n \, \rho(f)$  and $\rho
(h\circ f\circ h^{-1}) =\rho (f)$ hold  for  any  orientation preserving
circle homeomorphism  $h$. If $h$ is orientation reversing then $\rho (h\circ
f\circ h^{-1}) = -\rho (f)$.

The rotation number  gives rise to a  description of  the dynamical behavior
of circle homeomorphisms. Poincar\'{e} proved that:

\begin{Poincare}  Let f be an orientation preserving circle homeomorphism
with rotation number $\rho$. Then

(1) the rotation number  $\rho$ is rational if and only if
$f$ has a periodic point;

(2) if  the rotation number  $\rho$ is irrational, then $f$ is
semi-conjugate to $R_\rho$ the rotation by $\rho$, that is there
exists a continuous degree one monotone circle map $h$ such  that
$h\circ f = R_\rho\circ h$.
\end{Poincare}


The most natural generalization of circle homeomorphisms  are
 2-torus homeomorphisms isotopic to  identity.

Let   $\TT^2=\mathbb{R}^2/\mathbb{Z}^2$ be the  2-torus and
$\Pi\colon \mathbb{R}^2 \rightarrow \TT^2$ be  the natural
projection. Let $f\colon\TT^2 \rightarrow \TT^2$ be a continuous map
and $\widetilde{f}\colon\mathbb{R}^2 \rightarrow \mathbb{R}^2$ a
lift of $f$, that is, $f \circ \Pi= \Pi\circ\widetilde{f}$. If
$\widetilde{f}_1 $ and $\widetilde{f}_2$ are two lifts of $f$, it
holds that there exists $v\in \mathbb{Z}^2$ such that
$\widetilde{f}_1(\widetilde{x})= \widetilde{f}_2(\widetilde{x})+v$
for every $\widetilde{x} \in \mathbb{R}^2$, and if $f$ is isotopic
to identity then  for every $v\in \mathbb{Z}^2$ one has
 $\widetilde{f}(\widetilde{x}+v)= \widetilde{f}(\widetilde{x})+v$.
\medskip

In order to generalize the rotation number, we can consider the
sequences of 2-vectors $\{\frac { \tilde f
  ^n (\widetilde{x}) - \widetilde{x} } {n} \}_{n\in \mathbb N}$. But
these sequences may not converge and even in the case that the limit
exists, it may depend on the point $\widetilde{x}$. To avoid that
difficulty, Misiurevicz and Ziemian (\cite{MZ89}) have proposed to
define {\it a rotation set} as follows.

\bigskip

\subsection{Definitions of the rotation set and the rotation vectors.}

\medskip

Let  $f$ be a  2-torus homeomorphism
isotopic to identity and $\tilde f$ a lift of $f$, we call  {\bf ${\tilde
    f}$-rotation set} the subset of $\mathbb R ^2$ defined by

$$ \rho(\tilde f)= \bigcap_{i=1}^{\infty}\overline{\bigcup_{n\geq i} \left\{\frac {
    \tilde f   ^{n} (\widetilde{x}) - \widetilde{x}} {n}, \ \  \widetilde{x} \in \mathbb{R}^2\right\}}.$$

Equivalently, $(a,b) \in \rho(\tilde f)$ if and only if there exist
sequences $(\widetilde{x_i})$ with $\widetilde{x_i} \in
\mathbb{R}^2$ and $ n_i \to \infty$ such that
 $$(a,b)= \lim_{i\to \infty} \frac { \tilde f
  ^{n_i} (\widetilde{x_i}) - \widetilde{x_i}} {n_i}.$$

\medskip

Let $\widetilde x$  be in $\mathbb R^2$, the  {\bf $\widetilde{f}$-rotation
  vector   of $\widetilde x$} is the
2-vector defined by $\displaystyle \rho(\widetilde{f},\widetilde x)=\lim_{n\to
\infty} \frac { \tilde f
  ^{n} (\widetilde{x}) - \widetilde{x}} {n} \in \mathbb{R}^2$  if this limit exists.

\bigskip

\subsection {Some classical properties and results on  the rotation
set.}

Let $f$ be a  2-torus homeomorphism  isotopic to the identity and
 $\widetilde{f}$ be a lift of $f$ to $\RR^2$.

\medskip
\begin{itemize}
 \item Let $\widetilde x\in \mathbb R^2$ such that
   $\rho(\widetilde{f},\widetilde x)$ exists, it holds that
   \ \ $\rho(\widetilde{f},\widetilde x)\in \rho(\widetilde{f})$.

\item If $\widetilde{f}$ has a  fixed point   then  $(0,0) \in
\rho(\widetilde{f})$.

\item Misiurewicz and Ziemian (see  \cite{ MZ89} )have  proved that:
\begin{quote}
\begin{enumerate}

\item $\rho(\tilde f^n ) = n \rho(\tilde f)$

\item  $\rho(\tilde f +(p,q)) = \rho(\tilde f ) +(p,q) $,

\item the rotation  set is a  compact convex subset of
$\mathbb R^2$.

\end{enumerate}
\end{quote}

\item
 Franks (see \cite{F89}) proved that any rational point $q$ in
$\inte\rho(\tilde f )$ is the rotation vector of a lift of a
$f$-periodic point. That is there exists a $f$-periodic point $x \in
T^2$ and  a lift $\widetilde{x}$  of $x$, such that $\displaystyle
\lim_{n \to \infty}\frac { \tilde f
  ^n (\widetilde{x}) - \widetilde{x} } {n}=q $.

\item The rotation set is not a conjugacy invariant. However, if the
conjugating homeomorphism $h$ is isotopic to identity, then the homeomorphism
$\tilde f$ and  its conjugate homeomorphism $\tilde h\circ \tilde f \circ
\tilde h^{-1}$ have the same  rotation  set. Anyway, the property of having a
 lift with a rotation set of non empty interior does not depend on the choice
 of the lift and it is a  conjugacy invariant.

\end{itemize}
\bigskip

\subsection {Relationship between  the rotation set and the entropy.}

An  important conjugacy invariant is the {\bf topological entropy},
it can be defined for $f\colon X \to X$ as $\displaystyle h_{top} (f)
=\displaystyle \lim_{\epsilon \to 0} h_{top} (f,\epsilon )$,  where
$\displaystyle  h_{top} (f,\epsilon ) =  \limsup_{n \to +\infty}
\frac{1}{n} \log S (f,\epsilon,n)$ and $S (f,\epsilon,n ) $  is
the cardinality of a minimal $(n,\epsilon)$ spanning set (i.e a set
$E$  such that $\displaystyle X = \bigcup_{x\in E}
B_f(x,\epsilon,n)$, where $B_f(x,\epsilon,n)$ are dynamical balls).

\medskip

A result of Katok (\cite{Ka80}) claims that for  $C^ {1+\alpha}$ surface
diffeomorphisms the topological entropy is upper bounded by  the growth
rate of periodic points. Therefore,  any $C^{1+\alpha}$ surface
diffeomorphism without periodic points has null
topological entropy.

In \cite{LM91},  Llibre and MacKay  proved that any  toral
homeomorphism, isotopic to the identity and such that the interior of its
rotation set is not empty,  has positive topological entropy.

\bigskip

\subsection{ Remarks, questions and statement.}

The converse of this result by  Llibre and MacKay does not hold. We
will show examples where the rotation set has empty interior but the
topological entropy is positive (see Section 5 -examples
\ref{ex:ent}, \ref{ex:entro},  \ref{ex:diff} and \ref{ex:irred}).
So, we are interested in conditions implying  that the
interior of the rotation set is not empty. In his thesis
Kwapish (\cite {KwaT}) proved  that any Pseudo-Anosov homeomorphism  relative to a
finite set (in the sense of Handel) has rotation set with non empty
interior.  Our aim is to give dynamical conditions (in addition to positive
entropy) to obtain the same conclusion. So, we introduce the following definitions:

\medskip

\noindent {\bf Definitions.}

\begin{itemize}
\item A homeomorphism $f$ on $M$ is {\bf topologically
  transitive}  if there exists a point $x_0$ of $M$  such that the $f$-orbit of $x_0$  is dense.

\item   A homeomorphism $f$ on $M$ is {\bf totally
  transitive}  if any iterate of $f$, $f^n$, is topologically
  transitive.

\end{itemize}
\noindent Let $M$ be a manifold and $f$ a homeomorphism of $M$, it is denoted
by $(M,f)$.
\begin{itemize}
\item A subset $K$ of $M$ is \emph{ \bf essential}  if $K$ is not contained in a
disk  and there exists a finite  covering $M_N$ of $ M$ such that $
M_N\backslash P^{-1}(K)$ is not connected,  where $P:M_N \rightarrow M$ stands for the natural projection .

\item A homeomorphism $(M,f)$ is  \emph{\bf irreducible} if there is no
compact, $f$-invariant, of empty interior set which is essential.

\end{itemize}
 A non null homotopic circle is essential. In section 2, it will be showed that if $f$ admits a non null homotopic periodic circle then $f$ is not irreducible.

\newpage
 Our main result is the following:
\begin{thm}
\label{teo1} Let $f\colon\mathbb T^2 \rightarrow \mathbb T^2$ be a
diffeomorphism isotopic to identity satisfying the following
conditions:

\begin{enumerate}[{(1)}]
\item $f$ is of class $C^{1+ \alpha}$; \label{diff}
\item the  topological entropy  of  $f$ is  positive; \label{ent}
\item  $f$ is topologically transitive; \label{top}
\item  $f$ is irreducible; \label{irred}
\end{enumerate}
 then $\inte(\rho( \widetilde{f})) \neq \emptyset$, where $ \widetilde{f}$ is a lift of $f$ to $\RR^2$.
\end{thm}
%

\begin{rem}

The conditions \ref{diff} and  \ref{ent} will be  used for proving that there exists $n$ such that $f^n$
  admits a horseshoe.
  The conditions \ref{top} and  \ref{irred} will be  used for proving that there exists a suitable finite covering of $(\TT^2,f)$ that is totally transitive(it will be proved in lemma \ref{lemm:2-top}).
  Finally,
  according to our proof,  a torus homeomorphism
  isotopic   to identity, which admits a horseshoe and  admits a suitable finite covering that is totally transitive, has a rotation set with non-empty interior.

 Any lift to a finite covering of a Pseudo-Anosov map
  is also Pseudo-Anosov so it is totally transitive and admits
  horseshoe. As a consequence  we get, as it has been already proved by
Kwapisz, that  $\inte(\rho(\widetilde{f})) \neq \emptyset$.

\end{rem}%
%
%
In   Section 2 we give some properties related with irreducibility and
 explain how   it  arises in our context. Roughly
speaking, according to a result in  \cite{LM91},
the existence of  non null homotopic $f$-invariant circle
implies that the rotation set of $\widetilde{f}$ has empty interior. Hence, we  have to
avoid this case,  but  not only,   as other invariant sets can  play a similar
role (the pseudo-circles that arise in Anosov-Katok construction (see \cite{He86}), for example). We prove the main result in Section 3 using the total transitivity of $f$ and some finite covering, which is proved in Section 4. In Section
5 we exhibit different examples showing that the hypotheses are necessary.

\medskip

\begin{quote}

{\bf Acknowledgements.} We would like to thank Fran\c{c}ois Beguin, Sylvain Crovisier and Fr\'{e}d\'{e}ric Leroux
for interesting discussions and for their  encouragements. We are grateful to
Armando Treibich and  his family, without them this work would not have been possible.

\end{quote}

\section{Irreducibility and invariant circles}

\subsection{ Properties related with irreducibility }

\begin{prop}\label{prop1}
Let $f$ be a 2-torus homeomorphism then
\begin{enumerate}
\item If $f$ admits a non null homotopic periodic circle $C$ then $f$ is not irreducible
\item If $f$ is an  irreducible  homeomorphism  then any iterate
of $f$ is irreducible.
\item If $f$ is an  irreducible  homeomorphism  then any finite covering
of $f$ is irreducible.
\end{enumerate}
\end{prop}

\begin{proof}
\begin{enumerate}

\item

If $f$ admits a non null homotopic periodic circle $C$ then the
orbit of $C$ is a closed invariant subset $O_C$ and it has  empty
interior.  In  the double
covering of $M$ associated to $C$, the union of the lifts of $C$ and therefore  the union of the lifts
of $O_C$ is a disconnecting set.

\item

By contradiction, suppose that  there is an iterate $f^n$ of $f$ that is not irreducible. Then there exist a finite covering $M_N$ of the torus  and a compact set  $K$ of empty interior that is $f^n$-invariant,  such that $M_N\setminus P ^{-1}(K)=U_1 \sqcup U_2$, where $P: M_N \rightarrow \TT^2$ is the natural projection and $U_1$ and $U_2$ are disjoint non empty open sets.

Let $K'= \bigcup_{k=0} ^ {n-1} f^ k (K)$. This set is compact, of empty interior and $f$-invariant. The set $ P ^{-1}(K')$ contains  $ P ^{-1}(K)$, we write  $ P ^{-1}(K')=P ^{-1}(K) \cup L$, where $L$ is a compact of empty interior set. Then $M_N\setminus P ^{-1}(K')=M_N \setminus (P ^{-1}(K) \cup L )=  (U_1 \sqcup U_2)\setminus L =(U_1 \setminus L ) \sqcup (U_2\setminus L) $. Since $L$ is a closed set of empty interior, $M_N\setminus P ^{-1}(K')$ is the disjoint union of non empty open sets. This fact contradicts the irreducibility of $f$.

\item Suppose that there exist $(M_0,f)$ irreducible and $(M_1,f_1)$ a
finite lift of $(M_0,f)$ that is not irreducible. So there exists
a compact set $K_1$ that is $f_1$-invariant, of empty interior and essential. Let
$\pi_0 \colon M_1\to  M_0$ be the natural projection. We claim that the set
$K_0\colon=\pi_0(K_1)$ is  compact,  $f$-invariant,  of empty
interior and essential,  proving  that $(M_0,f)$ is not irreducible.
\\
We first prove that the set $K_0$ is $f$-invariant.
\\
Since $\pi_0\circ f_1=f\circ \pi_0$, we have $\pi_0\circ
f_1(K_1)=f\circ \pi_0(K_1)$. Thus $\pi_0(K_1)\supseteq f(\pi_0(K_1)$
that is $K_0 \supseteq f(K_0)$.
\\
The set $K_0$  is clearly compact and of empty interior because $\pi_0$ is a
local homeomorphism.
\\
It remains to prove that it is essential.\\
Denote by $M$ the finite covering of
 $(M_1,f_1)$  such that  $M\backslash\pi^{-1}_1(K_1)$
is not connected,  where $\pi_1 \colon M \to M_1$ is the natural
projection.  So  $M\backslash\pi^{-1}_1(K_1)$ can be written as the union of two disjoint open sets $A$ and $B$.
 Denote by $\pi \colon M \to M_0$ the natural projection and $D=\pi^{-1}(K_0)$. We have that $D=\bigcup\limits^n_{i=1}
\pi^{-1}_1(\gamma_i(K_1))$ where $\{\gamma_i\}$ stands for the finite
group consisting in the automorphisms of the covering $\pi_0$. It holds that $D$ is
closed and
has empty interior.\\
 So $M\backslash D=(A\cup B)\backslash D$ since $\pi^{-1}_1(K_1)\subset
 D$. Consequently,
  $M\backslash D =(A\backslash
D)\cup (B\backslash D)$ where  $A\backslash D$ and $B\backslash D$
are non empty open sets. Thus $M\backslash D$ is not connected, it remains to
prove that $K_0$ is not contained in a disk. If $K_0$ is included in a disk
$D_0$ then $K_1$ is contained in a finite union of disjoint disks, then we can
construct a disk $D_1$ that contains this union and therefore $K_1\subset
D_1$,  this contradicts the  irreducibility of $f$.

\end{enumerate}

\end{proof}

\subsection{Invariant circles }   Given a homeomorphism $f$ isotopic to the
identity on $\TT^2$, we are interested in relationships  between
the existence  periodic circles and the interior of the rotation set.

 In the case where $f$ admits a homotopically non-trivial invariant single
 curve, Llibre-Mac Kay (see \cite{LM91})  proved that  all the rotation
vectors  of  $f$-periodic points are collinear, therefore
the rotation set has empty interior. In the case where  $f$ admits a
homotopically non-trivial periodic single curve,  the rotation
set has empty interior since $\rho(\tilde f^n ) = n \rho(\tilde
f)$.

On the other hand, is  there a relationship between the existence of
homotopically trivial $f$-invariant single curve and the interior of
the rotation set  of $\tilde f$?

Let us show two examples:\begin{itemize}
\item
The identity map on $T^2$  fixes  every  circle  and its rotation set is $\{(0,0)\}$.
\item
In \cite{LM91}, the authors give examples of
$T^2$-homeomorphisms having rotation set of non empty interior. Let
us consider the particular example $f$  given by one of its lifts
$\widetilde{f}$ to $\mathbb{R}^2$. Let $\fr (x)=x -\lfloor x\rfloor$
be the fractional part of $x$ (where $\lfloor x\rfloor$ is the floor
function of $x$). We define $h,g\colon\mathbb{R}^2 \rightarrow
\mathbb{R}^2$ as $h(x, y)= (x,y+ \fr (2x))$ if $\fr (x) \in [0,\frac
12]$ and $h(x, y)= (x,y+ 2-\fr (2x))$ if $\fr (x) \in [\frac 12,1]$.
Analogously $g(x,y)=(x + \fr (2y), y)$ if $\fr (y) \in [0,\frac 12]$
and $g(x, y)= (x+ 2-\fr (2y),y)$ if $\fr (y) \in [\frac 12,1]$. Let
us define $\widetilde{f}=g\circ h$.
 It holds that its rotation set is $[0,1]^2$.
Actually, the point $(0,0)$ is fixed by $\widetilde{f}$, its
rotation vector is $(0,0)$, also $\widetilde{f}(\frac 12,0) =
(\frac12,0) +(0,1) $, $\widetilde{f}(0,\frac 12) = (0,\frac12)
+(1,0) $ and $\widetilde{f}(\frac 12,\frac 12) = (\frac12,\frac 12)
+(1,1) $. Then $\rho(\widetilde{f}, (\frac 12,0))=(0,1)$,
$\rho(\widetilde{f}, (0,\frac 12))=(1,0)$, and $\rho(\widetilde{f},
(\frac 12,\frac 12))=(1,1))$.

 We are going to modify this example in order that $f$  have
 an invariant homotopically trivial  circle and  that
$\widetilde{f}$ have still the same rotation set. Let us explain it.
The point $(0,0)$ is  fixed by $f$, we replace it by a small disk
$D$ by blowing up.  This construction does not change the rotation
set because of the following facts:
\begin{itemize} \item the points in $D$ have
the same rotation vector than $(0,0)$ ($D$ is $\widetilde{f}$-invariant)
\item  the three other vertices of the rotation set are
unchanged since they are realized by points for which the blow up did not
change the orbits.
\end{itemize}
On  the other hand, it
holds that $\partial D$ is a homotopically  trivial single curve
which is invariant by this  perturbation of $f$.

\end{itemize}

\begin{rem}
Since, there is no relationship between the existence of  homotopically
trivial $f$-invariant single curves and the interior of the rotation set of
$\widetilde{f}$, we ask for $K$ not to be contained in a disk, in the
definition of an essential set.
\end{rem}

\begin{rem}

There exist compact sets that are  $f$-invariant,  with
empty interior and  essential but   that are not circles
(they are not even locally
connected \cite{He86} and they are called pseudo-circles). It is
possible to change an invariant  non null-homotopic circle by a
connected invariant set that disconnects the torus and that  is not
locally connected. That is the reason why  --in the definition of
irreducibility--  we ask for the non existence of a compact, $f$-invariant, of
empty interior set (instead of a circle) which is essential.

\end{rem}

\begin{rem}
Other notions related to essential sets for surface homeomorphisms were proposed recently by T. J\"{a}ger
(called ''circloids'' in \cite{Ja}) and by A. Koropecki (called ''annular sets'' in \cite{Koro}).
\end{rem}

\section{Proof of  Theorem 1}
In this section we  prove Theorem \ref{teo1}.
\begin{proof}

Denote by  $\TT^2_*=\mathbb{R}^2/(2\mathbb{Z}) ^2$    the 4-1
covering of $\TT^2=\mathbb{R}^2/\mathbb{Z}^2$.

Let $F\colon \TT^2_*
\rightarrow \TT^2_*$ be a  lift of $f$ to $\TT^2_*$

Let the natural
projections be $\Pi\colon\mathbb{R}^2 \rightarrow \TT^2$,
$\Pi_*\colon\mathbb{R}^2 \rightarrow \TT^2_*$ and $P\colon \TT^2_* \rightarrow
\TT^2$. Note that $\Pi= P \circ \Pi_* $.  Let us  endow $\TT^2$
and $\TT_*^2$ with their usual flat Riemanian metrics (inherited from
the standard Euclididian metric on $\mathbb R^2$)  and the
associated  distances.

By hypotheses, $f $ is a $ C^{1+\alpha}$ diffeomorphism
and $h_{top}(f)>0$ therefore $F$ is also  a $ C^{1+\alpha}$
diffeomorphism  and $h_{top}(F)>0$.

By Katok (see \cite{Ka80}),   there exist $n \in \mathbb{N}$ and a
hyperbolic periodic point $y_0$ of $F$ with period $n$ such
that the intersection of the stable and unstable manifolds of $y_0$
is transversal. Thus,  there exists $k \in \mathbb{N}$ the minimal
positive number such that
 $F^{nk}(y_0)=y_0$ and  both eigenvalues of the
differential  $DF^{nk}(y_0)$ are positive. In what follows,
we denote  $F^{nk}$  by $f_*$.

Let us denote $x_0=P(y_0)$. Since $P$ is a local diffeomorphism,
$x_0$ is a hyperbolic fixed point of $f^{nk}$ and it has  the local
type of $y_0$.  Let $\widetilde{x_0} \in [0,1]\times [0,1]$ be a
lift of $x_0$ and $\tilde f_*$  be a lift of $f^{nk}$ to $\mathbb R^2$ such
that $\tilde f_*(\widetilde{x_0})=\widetilde{x_0}$. Note that $\tilde f_*$ is
a lift  of  both $f^{nk}$ and $f_*$.

We define $\widetilde x_1= \widetilde{x_0}$,  $\widetilde x_2=
\widetilde{x_0}+ (1,0)$, $\widetilde x_3= \widetilde{x_0}+ (0,1)$
and $\widetilde x_4= \widetilde{x_0}+ (1,1)$.
 Since $\tilde f_*$ is isotopic to identity, then $\tilde
 f_*(\widetilde{x_0}+(a,b))= \tilde f_*(\widetilde{x_0}) + (a,b)= \widetilde{x_0} +(a,b)$, for any $(a,b)\in
 \mathbb Z ^2$. Then every $\widetilde x_i$ is a fixed point of  $\tilde f_*$ and
therefore its   projection on $\TT^2_*$ denoted by
 $x_i$ is  fixed by  $f_*$. Note that there exists $i \in \{1,2,3,4\}$
 such that $y_0=x_i$.

\bigskip

{\bf Proposition 1.} {\it There exists  $0<\epsilon <\frac{1} {2}$
such that for $i\in \{2,3,4\}$ there exist $n_i \in \N$ and non
empty  compact sets   $L_i\subset  B_\epsilon (x_i)\subset \TT^2_*$
and $ L_1^i\subset B_\epsilon (x_1)$ such that  $L_i
=f_*^{n_i}(L_1^i)$ and  $P(L_1^i)= P(L_i)$.}

\medskip

{\bf Proof of  Proposition 1.} By the classical Hartman-Grobman's
theorem, there is an open subset $U$ of $\TT^2$ containing $x_0$
such that the restriction of $f^{nk}$ to $U$ is topologically
conjugated to its differential $Df^{nk}(x_0)$. By conjugating
$f^{nk}$ by a suitable  homeomorphism with support in a small
compact $K\supset U$, we may suppose that $f^{nk}$ is a linear
diagonal map in $U$ with eigenvalues $0<\lambda_1<1<\lambda_2$.

\smallskip

Fix  $\epsilon >0$ sufficiently small so that the ball $B_\epsilon (x_0)
\subset U$ and the lifts by $P$ of it  are disjoints
$\epsilon$-balls   $B_\epsilon (x_i)\subset \TT^2_*$, $i=1,...,4$ (this fact is
realized by taking $\epsilon<\frac{1}{2}$).
Restricted to  these balls, $f_*$ is a linear map.

\smallskip
Let   $x$ be in $ \TT^2_*$, we  denote by
$W^s (x)$ [resp. $W^u (x) $] the stable [resp. unstable] manifold of $x$ for
$f_*$.  For any $0<\delta\leq \epsilon$ and  $x\in\TT^2_*$, let's denote by
$W_\delta^s (x)$  [resp. $W_\delta^u (x) $]  the connected component of $W^s
(x) \cap  B_\delta (x)$ [resp.  $W^u (x) \cap   B_\delta (x)$] containing $x$.

\medskip

\noindent Since  $W^s (x_1) $ and $W^u (x_1)$ have a transverse intersection
in some  point $p$, there is:

-$N\in \mathbb N$ such that for $k\geq N$, $f_*^k(p) \in  W^s_\epsilon (x_1)$ (these
points converge monotonically  to $x_1$ when $k$ goes to $+\infty$) and

-$M\in \mathbb N$ such that  for $k\geq M$,  $f_*^{-k}(p) \in  W^u_\epsilon (x_1)$
(these points converge monotonically  to $x_1$ when $k$ goes to $+\infty$).

\smallskip

Consider a small  arc  $\Sigma^u$ of  $W^u (x_1) $  containing $f_*^N(p)$, the
segments  $f_*^k(\Sigma^u)$ for $k\geq N$ become larger and more  vertical with
$k$, so there is $ r\geq N$ minimal such that $d(x_1, f_*^r(p))\leq\frac
{\epsilon }{2}$ and  the arc $f_*^m(\Sigma^u)$ intersects transversally the
boundary of $ B_\epsilon (x_1)$ in two points.

Analogously, consider an arc  $\Sigma^s$ of  $W^s (x_1) $  containing $f_*^{-M}(p)$,
there is  $ r'\geq M$ minimal such that  $d(x_1, f_*^{-r'}(p))\leq\frac
{\epsilon }{2}$ and  the arc  $f_*^{-n'} (\Sigma^s)$ (almost
horizontal) intersects transversally the boundary of $ B_\epsilon (x_1)$ in
two points.

\smallskip

The  arcs $f_*^r(\Sigma^u)\cap  B_\epsilon (x_1)$ and
 $f_*^{-r'} (\Sigma^s)\cap  B_\epsilon (x_1)$ intersect  transversally.

\smallskip

Finally, we define a  rectangle $R_1$ in $\TT^2_*$   whose boundary is
the union of  arcs $C_{1,s}^j$ for $j=1,2$ included in $W^s(x_1)$
and  arcs $C_{1,u}^j$ for $j=1,2$ included in $ W^u(x_1)$. In fact
$x_1$ is a corner of $R_1$  and it is the intersection of the sides
$C_{1,s}^1$ and $C_{1,u}^1$ which are included in
$W^s_{\frac{\epsilon}{2}}(x_1)$ and $W^u_{\frac{\epsilon}{2}}(x_1)$
respectively, the two other sides are $C_{1,s}^2 \subset
f_*^m(\Sigma^s)\cap  B_\epsilon (x_1)$ and $C_{1,u}^2 \subset
f_*^{-m'}(\Sigma^u)\cap  B_\epsilon (x_1)$. By definition, the
diameter of $R_1$ is less than $\epsilon$.

\medskip

Let $R_0= P(R_1)$  be a rectangle in $\TT^2$. For $i=1\ldots 4$, denote by
$\gamma_i$ the automorphism of the finite  covering $P$ such that $\gamma_i
(x_1) =x_i$, and   let  $R_i=\gamma_i(R_1)$ be a rectangle in $\TT^2_*$ and for
$j=1,2$ we set $C_{i,u}^j=\gamma_i(C_{1,u}^j)$, $C_{i,s}^j=\gamma_i(C_{1,s}^j)$
  (see Figure 1).

\begin{figure}[h]\psfrag{f}{$f^{nk}
$}\psfrag{f4}{$f_*$}\psfrag{t2}{$\scriptstyle T^2
$}\psfrag{t24}{$\scriptstyle T^2_*$}\psfrag{p}{$\scriptstyle P$}
\psfrag{0}{$\scriptstyle x_0 $}\psfrag{1}{$\scriptstyle x_1$}
\psfrag{2}{$\scriptstyle x_3$}\psfrag{3}{$\scriptstyle x_4$}
\psfrag{4}{$\scriptstyle x_2$}\psfrag{r1}{$\scriptstyle R_1$}
\psfrag{r2}{$\scriptstyle R_3$}\psfrag{r3}{$\scriptstyle R_4$}
\psfrag{r4}{$\scriptstyle R_2$}\psfrag{c11}{$\scriptstyle
C^1_{1,u}$} \psfrag{c21}{$\scriptstyle C^2_{1,s}\subset
W^s_{f_*}(x_1)$}\psfrag{c21u}{$\scriptstyle C^2_{1,u}\subset
W^u_{f_*}(x_1)$} \psfrag{c11s}{$\scriptstyle C^1_{1,s} $}
\begin{center}\caption{Stable and unstable sides of $R_i$.}\includegraphics[scale=.4]{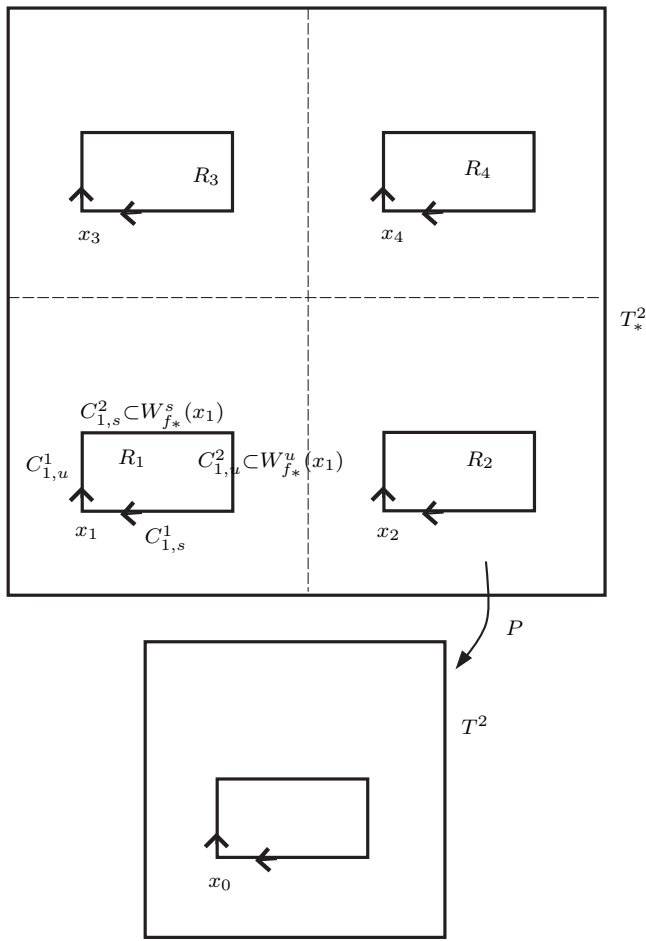}
\end{center}
\end{figure}
%
%

Fix $i\in
\{2,3,4\}$, as $f$ is  topologically transitive and irreducible then Corollary \ref{letoptran} implies that $f_*$ is topologically transitive so there exists $m_i \in \N$ such that $f_*^{m_i}(R_1) \cap \inte(R_i) \neq
\emptyset$.

\medskip

It is not possible that $f_*^{m_i}(R_1) \supset R_i$. In fact,
projecting via $P$ on $\TT^2$, we obtain that $f^{nkm_i}(R_0)
\supset R_0$.  According to the Hartman-Grobman Theorem,
the map  $f^{nk}$ has a unique fixed point in $R_0$  which is the
hyperbolic saddle point $x_0$.  Hence, it is not possible that
 $f^{nkm_i}(R_0) \supset R_0$.

 Furthermore, it is not possible that $x_i$ belongs to $ f_*^{m_i}(R_1) $ since $x_i$  is a $f_*$-fixed  point
and  $R_i$  and  $R_1$ are disjoint.

\medskip

As $W^u(x_1) \cap W^u(x_i)= \emptyset$ for $i\neq 1$ we
have that there exists $l_i\geq m_i$ such that
$f_*^{l_i}(C_{1,u}^1)\cap (C_{i,s}^1\cup C_{i,s}^2)\neq \emptyset$
and this intersection is topologically transversal.

\medskip

  There is no loss
of generality if we suppose that $f_*^{l_i}(C_{1,u}^1)\cap
C_{i,s}^1\neq \emptyset$, i.e $W^u(x_1) \cap W_{
\epsilon/2}^s(x_i)\neq \emptyset$.

\medskip

Since $W^u(x_1)$ is topologically transversal to $ W_{\epsilon/2}^s(x_i)$
we can  assert that  there exists $N_i$ such that at
least one  connected component of $f_*^{m}(C_{1,u}^1)\cap R_i$ has one end
point in $C_{i,s}^1$ and another one in $ C_{i,s}^2$ for all $m \geq
N_i$.

\medskip

Let us define $B_i$ as a small subrectangle in $R_i$ whose boundary
contains $C_{i,u}^1$ and stable and unstable arcs. We take the
stable sides of $B_1$ contained in the stable arcs of the boundary
of $R_1$. We choose the other unstable side of $B_1$, $ L_u$ close
enough to $C_{1,u}^1$ so that a connected component of
$f_*^{m}(L_{u})\cap R_i$ has one end point in $C_{i,s}^1$ and
another one in $ C_{i,s}^2$ for all $m \geq N_i$. Hence, one
connected component of $f_*^{m}(B_1)\cap R_i$ is a compact set  with
nonempty intersection with $C_{i,s}^1$ and with $ C_{i,s}^2$, for $m
\geq N_i$. Let $B_i \subset R_i$ verifying $P(B_i)= P(B_1)$.

\begin{figure}[h]\psfrag{f}{$f^{nk}
$}\psfrag{f4}{$f_*$}\psfrag{t2}{$T^2
$}\psfrag{t24}{$T^2_*$}\psfrag{p}{$P$} \psfrag{0}{$x_0
$}\psfrag{1}{$x_1$} \psfrag{2}{$x_2$}\psfrag{3}{$x_3$}
\psfrag{4}{$x_4$}\psfrag{r1}{$R_1$}
\psfrag{r2}{$R_2$}\psfrag{r3}{$R_3$}
\psfrag{r4}{$R_4$}\psfrag{c11}{$C^1_{1,u}\subset W^u_{f_*
\frac{\varepsilon }{2}}(x_1)$} \psfrag{c21}{$C^2_{1,s}\subset
W^s_{f_* }(x_1)$}\psfrag{c21u}{$C^2_{1,u}$}\psfrag{lu}{$L_u$}
\psfrag{pbi}{$P(B_3)$} \psfrag{c11s}{$C^1_{1,s}
$}\psfrag{fm1}{$f^{-m}_4(D_3)$} \psfrag{bi}{$B_3$}
\psfrag{b1}{$B_1$} \psfrag{di}{$D_3$}
\begin{center}\caption{The sets $B_i$ and $D_i$.}\includegraphics[scale=.4]{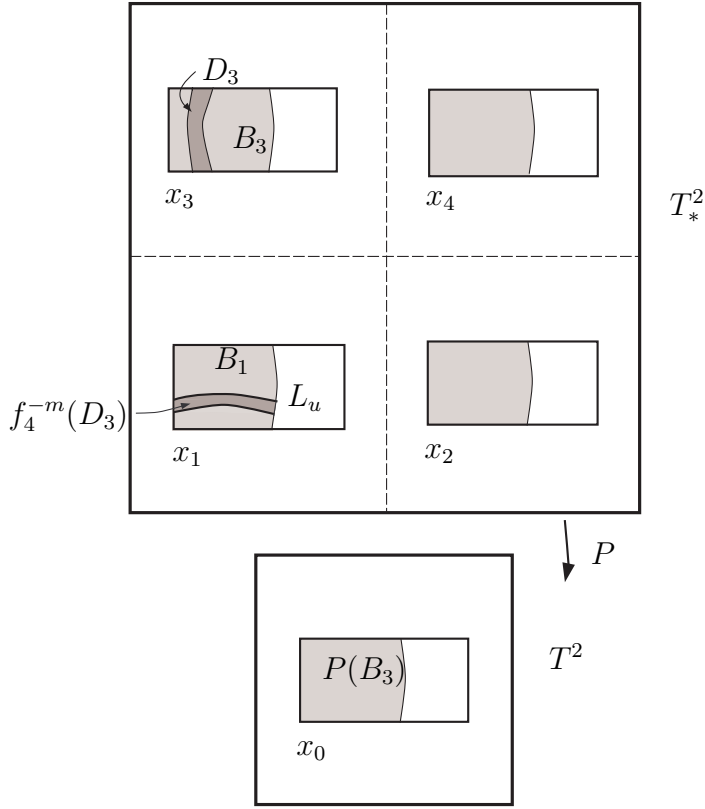}
\end{center}
\end{figure}%
%

\medskip

There exists $n_i\geq N_i$ such that for all $m\geq n_i$  one connected
component denoted by  $D_i$  of $f_*^{m}(B_1)\cap B_i$  is a compact
set included in $B_i$ with nonempty intersection with $C_{i,s}^1$
and with $ C_{i,s}^2$; and with empty intersection with the unstable
sides of $B_i$.   One  can  show that  the set $f_*^{-m}(D_i)$
is  connected,  compact and  contained  in $B_1$. It  intersects  the
unstable sides of $B_1$  and it  does not intersect  the stable
 sides of $B_1$ (see Fig 2).

\noindent It follows that   $\displaystyle \bigcap_{j=-N}^N f^{nkmj}(P(D_i))$
has the finite  intersection property.

\noindent Consequently,  the compact set (depending on $i$ and $m$)
in $\TT^2$ defined by:   $$\displaystyle
L=\bigcap_{j=-\infty}^{\infty} f^{nkmj}(P(D_i))$$ is  non empty,
$f^{mnk}-$invariant and it's contained in $R_0=P(R_i)\subset \TT^2$.

\smallskip

In what follows, we argue for $m=n_i$. For $j=1,...,4$, let $L^i_j=
P ^{-1}(L) \cap R_j$.  The set  $\displaystyle
\bigcup_{j=1}^4L_j^i$ is $f_*^{n_i}$-invariant. Moreover, since
$f_*$ is surjective and $L_1^i\subset f_*^{-n_i}(D_i)$, we have that
$f_*^{n_i}(L_1^i)=L_i^i$. Therefore, we have proved that
 there exist an integer  $n_i$ and  compact sets $L_1^i$ and $L_i:= L_i^i$
 such that   $P(L_i)=P(L_1^i)$ and  $f_*^{n_i}(L_1^i)=L_i$.   We get
 the proposition 1. \ \ \ \ \ \ $\blacktriangleleft$

\bigskip

\noindent {\bf Proof of the theorem.} We prove that the proposition
 1 implies the theorem.

Since $\widetilde{x_0}$ is a fixed point of $\tilde f_*$ it follows that
$\rho(\tilde f_*,\widetilde{x_0})=(0,0)$.

By proposition 1  for  $i=2$,  there exist $n_2$ and non empty
 compact sets $L_2\subset R_2$ and  $ L_1^2\subset R_1$ such that
 $P(L_2)=P(L^2_1)$ and   $f_*^{n_2}(L^2_1) = L_2$.

Let us denote ${\mathcal  L}_2$  [resp. ${\mathcal  L}^2_1$] a lift
of $L_2$ [resp. $L^2_1$] to $\mathbb R^2$.  Then there exist
$\mathbf{k_2} \in(2\mathbb{Z}) ^2$ such that $\tilde f_*^{n_2}({\mathcal
L}^2_1) = {\mathcal L}_2+\mathbf{k_2}$. Since $P(L_2)=P(L^2_1)$  and
$L_2 \subset  B_{x_2} (\varepsilon)$, we  have  that necessarily
${\mathcal L}_2 ={\mathcal L}^2_1+(1,0)$.  Therefore
$$
\tilde f_*^{n_2}({\mathcal L}_1^2) = {\mathcal L}_1^2+\mathbf{k_2}+(1,0).
$$

Hence,  $\tilde f_*^{ n_2}({\mathcal L}_1^2+\mathbf{k_2}+(1,0))= \tilde f_*^{
n_2} ({\mathcal L}_1^2)+ (\mathbf{k_2}+(1,0))=  {\mathcal L}_1^2 +
2(\mathbf{k_2}+(1,0))$ and  for every $k \in \mathbb{N}$
$$
\tilde f_*^{kn_2}({\mathcal L}_1^2)={\mathcal L}_1^2+ k (\mathbf k_2+(1,0)).
$$

Let $\widetilde{x} \in {\mathcal L}_1^2$.   For every $k$, there exists
$\widetilde{y_k} \in {\mathcal L}_1^2$ such that $\tilde f_*^{k n_2}(\widetilde{x})=
\widetilde{y_k} +k(\mathbf{k_2}+(1,0))$.
 It follows that
$$\frac{\tilde f_*^{k n_2}(\widetilde{x})-\widetilde{x}}{k n_2}=\frac{\widetilde{y_k}-\widetilde{x}}{k n_2}+\frac{k(\mathbf k_2+(1,0))}{k
n_2}$$ Then $$\lim_{k \rightarrow \infty} \frac{\tilde f_*^{k
n_2}(\widetilde{x})-\widetilde{x}}{k
n_2}=\frac{\mathbf{k_2}+(1,0)}{n_2}.$$ Hence,
$$\frac{\mathbf{k_2}+(1,0)}{n_2} \in \rho(\tilde f_*).$$
Analogously, for $i=3$ there exist integers ${n_3}$,  $k_3$ and a compact set
${\mathcal L}^3_1 $  in $\mathbb R^2$  such that
 $$\tilde f_*^{n_3}({\mathcal L}^3_1) = {\mathcal L}^3_1+\mathbf{k_3}+(0,1).$$

\noindent It comes that   $$\frac{\mathbf{k_3}+(0,1)}{n_3} \in \rho(\tilde f_*).$$

\medskip

Finally,  $(0,0) \in \rho(\tilde f_*)$  and the vectors
$\frac{\mathbf{k_2}+(1,0)}{n_2}$ and
$\frac{\mathbf{k_3}+(0,1)}{n_3}$ are linearly independent.

\medskip

Actually, for $i=2,3$  let us write $\mathbf{k_i} =
(2p_i, 2q_i)$  and  compute the determinant:
$$\det\left(n_2\frac{\mathbf{k_2}+(1,0)}{n_2},n_3
\frac{\mathbf{k_3}+(0,1)}{n_3}\right) = \left\vert \begin{array}{cc} 2p_2
+1 & 2p_3 \\2q_2 & 2q_3+1
  \end{array}\right\vert \ne 0, $$\noindent since it is the difference between  an  even number
and an odd number.

Then, it follows that $\rho(\tilde f_*)$ has 3 non colinear  points. By
convexity (see \cite{MZ91}) of $\rho(\tilde f_*)$,  we have that
$\inte(\rho(\tilde f_*))\neq \varnothing,$ for a lift $\tilde f_*$ of $f^{nk}$ to $\RR^2$. Thus,
this property holds  for any lift of $f^{nk}$ and therefore for any lift of $f$
to $\RR^2$. \end{proof}

\section{Proof of  the total transitivity.}

 Let $T^2_h:=\mathbb{R}^2/(2\mathbb{Z}\times \mathbb Z)$
 (resp. $T^2_v:=\mathbb{R}^2/(\mathbb Z \times 2\mathbb{Z})$) be a 2-1
 covering of $\TT^2=\mathbb{R}^2/\mathbb{Z}^2$, and let the natural projection be
$\Pi_h\colon T^2_h \rightarrow T^2$ (resp. $\Pi_v\colon T^2_v \rightarrow \TT^2$).
Let $f_h\colon T^2_h \rightarrow T^2_h$ be the lifting\ of $f$ to $T^2_h$
(resp. $f_v\colon T^2_v \rightarrow T^2_v$ be the lifting of $f$ to
$T^2_v$)

\begin{lem}\label{lemm:2-top}
Let $f\colon \mathbb T^2\to\mathbb T^2$ be a torus homeomorphism satisfying that  $f$ is topologically transitive and  irreducible
then
\begin{enumerate}[(1)]

\item the 2-1 coverings $f_h$ and $f_v$ are topologically transitive;
\item $f$ is totally transitive.

\end{enumerate}

\end{lem}

\begin{cor}
\label{letoptran}
 Let  $f\colon \mathbb T^2\to\mathbb T^2$ be a torus homeomorphism
         topologically transitive and irreducible. Let
  $F$ be the lift of $f$ defined in the proof of the theorem.
Then $F$  is totally transitive \end{cor}

\begin{proof}[Proof of the Corollary]
By definition, $F= (f_h)_v= (f_v)_h$. According to  the previous
lemma $f_h$ is topologically transitive and by the proposition
\ref{prop1} (3) it is irreducible, so we can apply once again this lemma
to $f_h$ and obtain that  $(f_h)_v$ is  totally
transitive.\end{proof}

\begin{proof}[Proof of lemma \ref{lemm:2-top}] \

\begin{enumerate}[(1)]

\item

We will argue by absurd for $f_h$. We suppose that $f$ is transitive
but not $f_h$.\\
Since $f$ is transitive, there exists $x_0$ such that $
O_f(x_0):=\{f^n(x_0) :n \in\mathbb Z\}$ is a dense set in $\TT^2$. Let
$\{x_1,x_2\}$ be the lifts of $x_0$ by $\pi^{-1}_h$. Let $\mathcal
O(x_i)$ be the $f_h$-orbit of $x_i$. We have that $\mathcal
O(x_1)\cup\mathcal O(x_2)=\pi^{-1}_h(O_f(x_0))$ is a dense set in
$T^2_h$ since $\pi_h$ is a local homeomorphism.  Since $f_h$ is not
transitive, neither $\mathcal O(x_1)$ nor $\mathcal O(x_2)$ is
dense. We claim that $\inte(\overline{\mathcal O(x_2)})\cap
\overline{\mathcal
  O(x_1)}=\varnothing$. Actually if there were a point $y$ in the intersection,
it would exist $m\in\mathbb Z$ such that $f_h^m(x_1)$ belongs to
$\overline{\mathcal
  O(x_2)}$ then $\mathcal O(x_1)\subset \overline{\mathcal O(x_2)}$. Thus
$\overline{\mathcal O(x_1)}\subset \overline{\mathcal O(x_2)}$ so we have $T^2_h=\overline{\mathcal O(x_1)}\cup\overline{\mathcal O(x_2)}= \overline{\mathcal O(x_2)}
$: a contradiction.\\
Analogously, the symmetric holds and these two equalities imply that
$\partial\overline{\mathcal  O(x_1)}=\partial\overline{ \mathcal  O(x_2)}$ and
$T^2_h=\inte (\overline{\mathcal O(x_1)})\sqcup \inte (\overline{\mathcal
  O(x_2)})\sqcup\partial \overline{\mathcal O(x_1)}$, where $\sqcup$
  denotes disjoint union.
The set $\partial\overline{\mathcal O(x_1)}$ is a closed invariant
of empty interior subset that disconnects $T^2_h$.

\medskip

 We are going to prove that it can
not be contained in a disk $D\subset T_h$.

\smallskip

\noindent  Suppose, by absurd that
$\partial\overline{\mathcal O(x_1)}\subset D$.  First, we prove that $\inte \overline{\mathcal O(x_1) }$ or
$\inte \overline{\mathcal O(x_2)}$ is
  included in $D$. In fact,
 if both of them intersect the complement $ D ^c$
  of $D$, we can take  a path  in $ D ^c$  joining a point of $\inte
  \overline{\mathcal O(x_1)}$ and a point of $\inte
  \overline{\mathcal O(x_2)}$. By connexity,  this path must contain a point
    of the boundary  $\partial\overline{\mathcal O(x_1)}$, which contradicts
    the fact that $\partial\overline{\mathcal O(x_1)}\subset D$.

Finally, suppose  that  $\inte \overline{\mathcal O(x_1)}\subset D $ then
 $\inte
  \overline{\mathcal O(x_2)}\supset  D^c$, but this is not possible since
  $\inte
  \overline{\mathcal O(x_1)}$ and  $\inte \overline{\mathcal O(x_2)}$
are homeomorphic.

We have proved that $\partial\overline{\mathcal O(x_1)}$ is a closed invariant
of empty interior subset that disconnects $T^2_h$ and that  is not contained
in a disk. But this is a contradiction with  the fact that $f$ is  irreducible.

\item
We suppose that $f$ is transitive
but there exists a positive integer $N$  such that  $f^N$ is not.\\
So, there exists $x_0$ such that $
O_f(x_0)$ is a dense set in $\TT^2$, but $O_{f^N}(x_0)$ is not.
We have $$\overline{O_{f^N}(x_0)}\cup \overline{O_{f^N}(f(x_0))}\cup...\cup\overline{O_{f^N}(f^{N-1}(x_0))}=\overline{O_{f}(x_0)}=\TT^2.$$
As in the proof of the first item of this lemma we have that if $ int(\overline{O_{f^N}(f^i(x_0))})\cap \overline{O_{f^N}(f^j(x_0))}\neq \emptyset$ then $\overline{O_{f^N}(f^i(x_0))}\supseteq \overline{O_{f^N}(f^j(x_0))}$.

Let $K=\min\{i=1,...,N-1 : int(\overline{O_{f^N}(x_0)})
\cap \overline{O_{f^N}(f^i(x_0))}\neq \emptyset \}$.
 Then $\TT^2$ can be decomposed as  union of closed sets with disjoint interiors:$$ \TT^2= \overline{O_{f^N}(x_0)}\cup \overline{O_{f^N}(f(x_0))}\cup...\cup\overline{O_{f^N}(f^{K-1}(x_0))}.$$
 In the case that $K$ is an even number then $ \TT^2= A\cup B$, where $A= \overline{O_{f^N}(x_0)}\cup \overline{O_{f^N}(f(x_0))}\cup...\cup\overline{O_{f^{N}}(f^{\frac{K}{2}}(x_0))}$ and $B=\overline{O_{f^{N}}(f^{\frac{K}{2}+1}(x_0))} \cup...\cup\overline{O_{f^N}(f^{N-1}(x_0))}.$ The sets $A$ and $B$ have disjoint interiors, their boundaries coincide and they are $ f^N$- invariant. Thus $ \TT^2= int(A)\sqcup int(B)\sqcup \partial B $. According to proposition \ref{prop1} $f^N$ is irreducible, then $\partial B$ is contained in a disk. As $A$ and $B$ are homeomorphic, an analogous proof to the previous item shows that this can not occur.

  In the case that $K$ is an odd number then $ \TT^2= A\cup B \cup C$, where $A= \overline{O_{f^N}(x_0)}\cup \overline{O_{f^N}(f(x_0))}\cup...\cup\overline{O_{f^{N}}(f^{\frac{K-1}{2}}(x_0))}$ and $B=\overline{O_{f^{N}}(f^{\frac{K+1}{2}}(x_0))} \cup...\cup\overline{O_{f^N}(f^{N-2}(x_0))},$ and $C=\overline{O_{f^N}(f^{N-1}(x_0))}$. The sets $A$ and $B\cup C$ have disjoint interiors, their boundaries coincide and they are $ f^N$- invariant. Thus $ \TT^2= int(A)\sqcup int(B \cup C)\sqcup \partial A $. According to proposition \ref{prop1} $f^N$ is irreducible, then $\partial A$ is contained in a disk $D$.

By connexity, as in the previous item, we prove that $A$ or
$B\cup C$ is  included in $D$.

Finally, if $ A\subset D $ then $  B\subset f^{\frac{N+1}{2}}(D)$ and $ C\subset f^{N-1}(D)$.

 If $ B\cup C\subset D $ then $  A\subset f^{-\frac{N+1}{2}}(D)$.

Then, in both cases, the torus is the union of  at most three disks which is impossible.

So we have proved that $f$ is totally transitive.

\end{enumerate}

\end{proof}

\begin{rem}
The topological transitivity of $f$ is not enough to guarantee that
$f_h$ (or $f_v$) is topological transitive. For example, consider a
diffeomorphism $f$ of $\mathbb{A}= [0,1]\times S^1$ obtained by the
Katok-Anosov process (see \cite{AnKa}) in such a way that:
\begin{itemize}
\item $f$ is topologically transitive in $
\inte(\mathbb{A})=(0,1)\times S^1$, and
\item $f(0,x)=f(1,x)$ for all $x\in S^1$.

\end{itemize}

 We collapse the circles $\{0\}\times S^1$ and $\{1\}\times S^1$ by
 identifying $(0,x)$ with $(1,x)$ for $x\in S^1$. Then, we have a
 diffeomorphism of the torus, $\widehat{f}$, verifying that
 $\widehat{f}$ is topologically transitive on $T^2$ and $C=\{0\}
 \times S^1$ is a circle invariant of $\widehat{f}$.
 Let us consider the finite covering $T^2_h$ of the torus $T^2$ and
let $\widehat{f}_h$ be
  the lifting of $\widehat{f}$ to $T^2_h$.
  The lifting of $C$ is the union of two circles $C_1$ and $C_2$ that disconnect
  $T^2_h$, and the $T^2_h \backslash (C_1 \cup C_2)$ is the disjoint
  union of two cylinders.
  The orbits of $\widehat{f}_h$  are dense in each cylinder but $\widehat{f}_h$
  is not topologically transitive.
\end{rem}

\section{Examples}
In this section, we give  examples in order to show that each hypothesis of the
theorem is necessary.

\begin{enumerate}[{1)}]

\smallskip

\item  Missing hypothesis \ref{ent}.

Let $R_{\alpha,\beta}$ be the rotation of vector $(\alpha, \beta)$ with $\alpha$ and
$\beta$ irrational, that is, the projection to $T^2$ of the
translation of vector $(\alpha, \beta)$ in $R^2$. It is a well known
fact that $R_{\alpha,\beta}$ is topologically transitive, it is
differentiable and it is irreducible, but its rotation set is
$\{(\alpha,\beta)\}$. This example shows that conditions \ref{diff},
\ref{top}, and \ref{irred} of the theorem \ref{teo1} do not ensure
that the interior of the   rotation set is not empty.

\medskip

\item Missing hypothesis \ref{top} and \ref{irred}.

\label{ex:ent}

Let  $f_D\colon \overline{\mathbb D^2} \to\overline{\mathbb D^2}$ be a
diffeomorphism such that there exists a horseshoe in the interior of
$\mathbb D^2$  and such that $f_D$ is the identity on $\partial
\overline{\mathbb D^2}$. It follows that $f_D$ has  positive
entropy. Let us embed $\overline{\mathbb D^2}$ in $\mathbb T^2$ and
then  extend $f_D$  to $f$ by the identity on  $\mathbb
T^2\backslash \overline{\mathbb D^2}$. It holds that $f$ has
positive entropy and the rotation set has empty interior(because
there exist invariant circles homotopically non trivial). This
example show that conditions \ref{diff} and \ref{ent} do not ensure
that the conclusion of the theorem is verified.

\medskip

\item

\label{ex:entro}

Missing hypothesis  \ref{top}.

We start with an irrational flow $\phi^t_0$ on $T^2$. By making an
appropriate smooth time change  vanishing at one point $x_0$ (we
replace the vector field $X$ by $g.X$ where $g(x_0) =0 $ and
$Dg(x_0) =0 $), we get a new smooth topologically transitive flow
$\phi^t_g$  with a fixed point $x_0$. Consider the time one map of
this flow, $f$, and replace $x_0$ by a small closed disk $D_0$ by
blowing up. The dynamic of the blow up of $f$ on  $\partial D_0$ is
of the type north-south. We have that $D_0 \backslash \{ N, S\}$ is
foliated by meridians $\{M_t\}_{t\in [-1,1]}$ and $\partial
D_0=M_{-1} \cup M_1 \cup \{ N, S\}$. Let $\gamma\colon D_0 \rightarrow
D_0$ a differentiable map such that $\gamma |M_i= f|M_i$, for
$i=-1,1$, $\gamma(N)=N$, $\gamma (S)=S$, for all $t\in [-1,1]$ $M_t$
is $\gamma$-invariant and $\gamma|M_t=Id$ for $t \in [-\frac 12,
\frac 12]$. In the blow up manifold, $T^2$ we define the
differentiable map $\Gamma$ as $\Gamma(x)= f(x)$ if  $x\in T^2 / D_0
$ and $ \Gamma (x)=\gamma(x) \mbox{ if } x\in D_0$. Let $D_1=
\cup_{t\in [- \frac 12, \frac 12]} M_t \cup  \{ N, S\}$, it holds
that $\Gamma |D_1= Id$.

 As in the previous example, we can put a horseshoe
in the interior of $D_1$.
The resulting diffeomorphism satisfies trivially the
conditions \ref{diff},  \ref{ent}.

 It also verifies the condition   \ref{irred}. Moreover,  an invariant compact
  set $K$  of $\Gamma$ is  included  either:

\begin{itemize}

\item in $D_0$, in this case $K$ is not essential or

\item  in the complement $D_0^c $ of $D_0$, in this case $K$  coincides with
 $D_0^c $  (since each  orbit  in $D_0^c $ is dense in it) hence  its interior is
 not empty.

\end{itemize}

 But it does  not satisfy the condition   \ref{top}, since it
has an invariant disk.

Finally,  its rotation set has empty interior. In fact, before the blowing
up, the  map $f$ is the time one map of a flow with a fixed point $x_0$ so according to
Franks and Misiurewicz's result (see \cite{FM}) its rotation set is a line
segment containing $(0,0)$. The blowing up  does not change the
rotation set because of the following facts:
\begin{itemize} \item the points in $D_0$ have
the same rotation vector  as  $x_0$  which is $(0,0)$ ($D_0$ is $\Gamma$-invariant),
\item   for  the  points out of $D_0$, the blowing up does not
change the orbits so it does not change their rotation vectors.
\end{itemize}

\medskip

 \item Missing  hypothesis  \ref{diff}.

According to \cite{Rees} there exists  a torus homeomorphism $f_0$
  isotopic to the identity  such that it is minimal
  and it has positive entropy. Since $f_0$  is minimal, all its orbits are dense so it has no periodic points.
   By \cite{F89} we know that if the interior of the rotation set is not empty, then
   each  vector with rational coordinates in
    the interior of the rotation set is realized as the rotation vector of a periodic
    point. It follows that $f_0$ verifies that
    $\inte(R(f_0))=\varnothing$.
    This example shows that conditions \ref{ent}, \ref{top} and
    \ref{irred} are not enough to guarantee that the interior of the
    rotation set is not empty.
\label{ex:diff}

\item  Missing  hypothesis  \ref{irred}

According to \cite{Kato} there exists  a $C^{\infty}$ topologically transitive Bernoulli
diffeomorphism $f_0\colon S^2 \to S^2$ which preserves a smooth
positive measure on $S^2$. Since $f_0$ (or $f_0^2$) preserves
orientation then it is isotopic to the identity.

As in the construction of \cite{Kato}, there exist $x_1, x_2$, two
fixed points of $f_0$ ( or $f_0^k$) such that $D f_0(x_i)=Id, \
i=1,2.$ We can replace $x_1$ and $x_2$ by small closed disks $D_1$
and $D_2$, respectively, by blowing up. The dynamic of the blow up
of $f_0$ on $\partial D_1$ and $\partial D_2$ is the identity. By
gluing $\partial D_1$ and $\partial D_2$ we have a smooth map
$f\colon T^2 \to T^2$ which is topologically transitive and it has
positive  entropy but there exists a compact $f$-invariant of empty
interior set ($\partial D_1$) which is essential. This example fails
to be irreducible because of the existence of  a non null homotopic
 invariant circle, then its rotation set has  empty interior because
 of Llibre and Mac
Kay's result
 ( see  \cite{LM91}).

\label{ex:irred}

\end{enumerate}
%
%
%
%
%


\bibliographystyle{alpha}

\end{document}